\documentclass[preprint, 11pt, english]{elsarticle}
\usepackage{amsmath}
\usepackage{latexsym, amssymb}
\usepackage{amsthm}
\usepackage{txfonts,xcolor}

\newtheorem{thm}{Theorem}[section] 

\newtheorem{prop}[thm]{Proposition}

\theoremstyle{definition}
\newtheorem{rem}[thm]{Remark}
\newtheorem{exmpl}[thm]{Example}

\newtheorem{algo}[thm]{Algorithm}

\newcommand\operA[2]{{\if!#2!\operatorname{#1}\else{\operatorname{#1}_{#2}^{\phantom{I}}}\fi}} 

\newcommand\Cref[1]{{Corollary~\ref{#1}}}%

\def\tr{{\operatorname{Tr}}}

\def\norm{{\operatorname{Norm}}}

\newcommand{\Trace}[1][]{\if!#1!\operatorname{Tr}\else{\operatorname{Tr}_{#1}^{\phantom{I}}}\fi} 

\long\def\forget#1\forgotten{{}} %

\def\({\left(}
\def\){\right)}

\newif\iffurther
\furtherfalse

\newif\ifXY 
\XYtrue     
\ifXY

\input xy
\input xyidioms.tex
\usepackage{xy}
\xyoption{all} %
\fi 

\usepackage{babel}

\journal{TBD}

\begin{document}

\begin{frontmatter}

\title{Factoring Octonion Polynomials}
\author{Adam Chapman}
\ead{adam1chapman@yahoo.com}
\address{Department of Computer Science, Tel-Hai Academic College, Upper Galilee, 12208 Israel}

\begin{abstract}
We provide an analogue of Wedderburn's factorization method for central polynomials with coefficients in an octonion division algebra, and present an algorithm for fully factoring polynomials of degree $n$ with $n$ conjugacy classes of roots, counting multiplicities.
\end{abstract}


\end{frontmatter}

\section{Introduction}

For fields $F$ it is well-known that a polynomial $f(x) \in F[x]$ factors into linear factors $f(x)=c(x-\lambda_1)\cdot \dots \cdot (x-\lambda_n)$ if and only if it is has $n $ roots in $F$ (counting multiplicities), in which case the roots are precisely $\lambda_1,\dots,\lambda_n$.

For associative division rings $D$ the statement was generalized by Wedderburn in \cite{Wedderburn:1921} (see also \cite{Rowen:1992} and \cite{LamLeroy:2004}).
It is important to note that by ``polynomials" over any algebra in this context, we mean \textit{central} polynomials, i.e., polynomials in which the indeterminate is assumed to commute with any element in the ring.
More precisely, the ring of (central) polynomials $D[x]$ over a division ring $D$ is defined to be the scalar extension $D \otimes_F F[x]$ of $D$ to the ring of polynomials $F[x]$ over the field $F=Z(D)$. For any element $\lambda \in D$, the substitution map $S_\lambda : D[x] \rightarrow D$ is defined by sending each polynomial $f(x)=c_n x^n+\dots+c_1 x+c_0$ to $f(\lambda)=c_n \lambda^n+\dots+c_1 \lambda+c_0$, thus placing the coefficients on the left-hand side of the indeterminate before plugging the value $\lambda$ in $x$, and therefore these polynomials are also known as ``left" polynomials.

What Wedderburn showed was that if $\lambda_1 \in D$, $f(x)\in D[x]$ and $\lambda_1$ is a root of $f(x)$, then $f(x)=g(x)(x-\lambda_1)$ for some polynomial $g(x) \in D[x]$. Then, if $\lambda_2$ is another root of $f(x)$ in $D$, $(\lambda_1-\lambda_2)\lambda_2(\lambda_1-\lambda_2)^{-1}$ is a root of $g(x)$. This is called ``Wedderburn's method" in \cite{Rowen:1992}.

Wedderburn's main motivation was the following observation:
Let $F$ be a field and $D$ a central division $F$-algebra. Suppose $f(x)$ is a monic irreducible polynomial in $F[x]$.
It is possible that $f(x)$ becomes reducible over $D[x]$.
In particular, if $f(x)$ has a root $\lambda_1$ in $D \setminus F$, then $f(x)$ decomposes in $D[x]$ as $(x-\lambda_n)(x-\lambda_{n-1})\dots (x-\lambda_1)$ where $\lambda_1,\dots,\lambda_n$ are inner conjugates in $D$.

Our goal in this work is to study the analogous situation for octonion division algebras.
Since by Kleinfeld's theorem (\cite[Chapter 7, Section 3]{SSSZ}), every nonassociative alternative division algebra is an octonion algebra (see \cite{SpringerVeldkamp} and \cite[Section 33]{BOI} for reference), this will complete the picture for alternative division algebras in full generality.
The main two results are thus the following:
\begin{enumerate}
\item An analogue of Wedderburn's method for the ring of (central) polynomials $A[x]$ over an octonion division algebra $A$.
\item An algorithm for fully factoring polynomials $f(x)$ of degree $n$ in $A[x]$ iteratively into linear factors $f(x)=((\dots(c(x-\lambda_n))\dots(x-\lambda_3))(x-\lambda_2))(x-\lambda_1)$, under the assumption there are $n$ conjugacy classes for the roots of $f(x)$, counting multiplicity.
\end{enumerate}

\section{Octonion Polynomials}

Let $A$ be an octonion division algebra with center $F$.
Recall that such an octonion algebra is a Cayley-Dickson doubling of a quaternion algebra $Q$ over $F$, i.e., $A=Q \oplus Q \ell$ as an $F$-vector space, and the multiplication is given by
$$(q+r\ell)\cdot(s+t\ell)=qr+\gamma \overline{t} r+(r \overline{s}+tq)\ell$$
for any $q,r,s,t \in Q$, where $\gamma$ is some fixed element in $F^\times$, and $z \mapsto \overline{z}$ is the canonical involution on $Q$.
The quaternion algebra $Q$ in turn admits the structure
$$Q=F \langle i,j : i^2=\alpha, j^2=\beta, i j=-j i \rangle$$
for some $\alpha,\beta \in F^\times$ when $\operatorname{char}(F) \neq 2$, in which case the canonical (symplectic) involution is given by 
$$\overline{a+bi+cj+dk}=a-bi-cj-dk,$$
and 
$$Q=F \langle i,j : i^2+i=\alpha, j^2=\beta, j i j^{-1}=i+1 \rangle$$
for some $\alpha \in F$ and $\beta \in F^\times$ when $\operatorname{char}(F)=2$, in which case the canonical involution is given by 
$$\overline{a+bi+cj+dk}=a+b+bi+cj+dk.$$
The canonical involution extends from $Q$ to $A$ by
$$\overline{q+r\ell}=\overline{q}-r\ell.$$
This involution gives rise to the trace and norm maps
$$\tr : A \rightarrow F$$
$$\norm : A \rightarrow F$$
defined by $\tr(z)=z+\overline{z}$ and $\norm(z)=\overline{z}\cdot z$.
Every element $z \in A$ then satisfies $z^2-\tr(z)\cdot z+\norm(z)=0$, and thus the algebra $A$ is a quadratic algebra.
Moreover, it is a composition algebra, for its norm form is multiplicative (see \cite[Section 33]{BOI}).
The algebra $A$ is also alternative, which means that every two elements $a,b$ in $A$ live in an associative subalgebra, and as a result, conjugation is well-defined.

Following the literature on (central) polynomials over associative division algebras, by the ring of polynomials $A[x]$ we mean the tensor product $A \otimes_F F[x]$ of $A$ and the ring of polynomials $F[x]$ over $F$.
The center of $A[x]$ is thus $F[x]$.
Therefore, every polynomial $f(x) \in A[x]$ can be written in the ``standard form" $f(x)=c_n x^n+\dots+c_1 x+c_0$ where $c_0,\dots,c_n \in A$, in which case the coefficients are placed on the left-hand side of the indeterminate.
For each $\lambda \in A$, the substitution map 
$$S_\lambda : A[x] \rightarrow A$$
$$f(x)\mapsto f(\lambda)$$
is defined by sending $f(x)=c_n x^n+\dots+c_1 x+c_0$ to $f(\lambda)=c_n \lambda^n+\dots+c_1 \lambda+c_0$.
The element $\lambda \in A$ is thus a ``root" of $f(x)$ when $f(\lambda)=0$.

We start with factoring out one linear factor given a root:
\begin{prop}
Given a polynomial $f(x) \in A[x]$ and an element $\lambda \in A$, $f(x)$ decomposes as $f(x)=g(x)(x-\lambda)$ if and only if $\lambda$ is a root of $f(x)$, in which case
$$g(x)=d_{n-1} x^{n-1}+\dots+d_1 x+d_0, \ \text{where}$$
$$d_k=c_n \lambda^{n-1-k}+c_{n-1} \lambda^{n-2-k}+\dots+c_{k+1}, \ k\in \{0,\dots,n-1\}.$$
\end{prop}

\begin{proof}
Consider the expression $c_n x^n+\dots+c_1 x-(c_n \lambda^n+\dots+c_1 \lambda)$.
On the one hand, since $\lambda$ is a root of $f(x)$, we have $f(\lambda)=0$, i.e., $c_n \lambda^n+\dots+c_1 \lambda+c_0=0$, and therefore $c_0=-(c_n \lambda^n+\dots+c_1 \lambda)$, which means that $f(x)=c_n x^n+\dots+c_1 x-(c_n \lambda^n+\dots+c_1 \lambda)$.

On the other hand, for each $k \in \{1,\dots,n\}$, the elements $c_k$ and $\lambda$ live in an associative subalgebra $D$ of $A$, and thus together with $x$ they live in an associative subring $D[x]$ of $A[x]$. We therefore  have
$$c_k x^k-c_k \lambda^k=c_k(x^k-\lambda^k)=c_k(x^{k-1}+\lambda x^{k-2}+\dots+\lambda^{k-1})(x-\lambda)$$
$$=(c_k x^{k-1}+c_k \lambda x^{k-2}+\dots+c_k \lambda^{k-1})(x-\lambda).$$
Hence, by rewriting $c_n x^n+\dots+c_1 x-(c_n \lambda^n+\dots+c_1 \lambda)$ as $(c_n x^n-c_n \lambda^n)+(c_{n-1} x^{n-1}-c_{n-1}\lambda^{n-1})+\dots+(c_1 x-c_1 \lambda)$, we obtain that the latter is equal to $g(x)(x-\lambda)$ where $g(x)=d_{n-1}x^{n-1}+\dots+d_1 x+d_0$ and for each $k \in \{0,\dots,n-1\}$, $d_k=c_n \lambda^{n-1-k}+c_{n-1} \lambda^{n-2-k}+\dots+c_{k+1}$.
\end{proof}

\begin{rem}
The analogue of Wedderburn's theorem \cite[Theorem 0.4]{Rowen:1992} holds true also in the octonionic case:
let $A$ be an octonion division algebra with center $F$, a field, and let $f(x)$ be a monic irreducible polynomial of degree $n$ in $F[x]$.
Suppose it has a root $\lambda_1$ in $A$.
Then, since $A$ is an octonion algebra, $\lambda$ lives inside a quaternion $F$-subalgebra $D$ of $A$. 
Now, by Wedderburn's theorem, $f(x)$ decomposes as $(x-\lambda_n)\cdot \dots \cdot (x-\lambda_1)$ in $D[x]$, where $\lambda_1,\dots,\lambda_n$ are inner conjugates in $D$.
Since $D[x]\subseteq A[x]$, the same decomposition applies also in $A[x]$.
\end{rem}

We present here the direct generalization of Wedderburn's method in the alternative case:
\begin{prop}\label{Wedderburn}
Given a polynomial $f(x)$ in $A[x]$, if $\lambda_1$ and $\lambda_2$ in $A$ are two distinct roots of $f(x)$, then $\gamma=(\lambda_1-\lambda_2) \lambda_2 (\lambda_1-\lambda_2)^{-1}$ is a root of the polynomial
$$h(x)=d_{n-1}x^{n-1}+\dots+d_1 x+d_0, \ \text{where}$$
$$d_k=\gamma^k\left( (\gamma^{-k}((\lambda_1-\lambda_2)^{-1}c_n)) \lambda_1^{n-1-k}+\dots+\gamma^{-k} ((\lambda_1-\lambda_2)^{-1}c_{k+1})\right), \ k\in \{0,\dots,n-1\}.$$
\end{prop}

\begin{proof}
Since both $\lambda_1$ and $\lambda_2$ are roots of $f(x)$,
we have
$$c_n \lambda_1^n+\dots+c_1 \lambda_1+c_0=0, \ \text{and}$$
$$c_n \lambda_2^n+\dots+c_1 \lambda_2+c_0=0.$$
Subtracting the second from the first gives
$$c_n(\lambda_1^n-\lambda_2^n)+\dots+c_1(\lambda_1-\lambda_2)=0.$$
Now, for each $k \in \{1,\dots,n\}$,
$$\lambda_1^k-\lambda_2^k=(\lambda_1^{k-1}+\lambda^{k-2}\gamma+\dots+\gamma^{k-1})(\lambda_1-\lambda_2)$$
where $\gamma=(\lambda_1-\lambda_2)\lambda_2(\lambda_1-\lambda_2)^{-1}$.
For each $k$, write $c'_k=(\lambda_1-\lambda_2)^{-1}c_k$, and then 
$$c_k(\lambda_1^k-\lambda_2^k)=(\lambda_1-\lambda_2)(c_k'(\lambda_1^{k-1}+\lambda^{k-2}\gamma+\dots+\gamma^{k-1}))(\lambda_1-\lambda_2)^{-1}$$
by the Moufang identity $(r s)(t r)=r(st)r$.
It remains to note that 
$$c_k'(\lambda_1^{k-1}+\lambda^{k-2}\gamma+\dots+\gamma^{k-1})=c_k'\lambda_1^{k-1}+\gamma((\gamma^{-1}c_{k-1}')\lambda_1^{k-2})\gamma+\dots$$
$$+\gamma^{k-2}((\gamma^{2-k}c_1')\lambda_1)\gamma^{k-2}+\gamma^{k-1}(\gamma^{1-k}c_0')\gamma^{k-1}$$
and then the sum $c_n(\lambda_1^n-\lambda_2^n)+\dots+c_1(\lambda_1-\lambda_2)$ turns into
$$(\lambda_1-\lambda_2)\left(d_{n-1}\gamma^{n-1}+\dots+d_1 \gamma+d_0 \right)(\lambda_1-\lambda_2)^{-1},$$
which means $d_{n-1}\gamma^{n-1}+\dots+d_1 \gamma+d_0=0$.
\end{proof}

Note that when all the coefficients $c_n,\dots,c_0$ and the roots $\lambda_1$ and $\lambda_2$ belong to an associative subalgebra, the polynomial $h(x)$ from Proposition \ref{Wedderburn} coincides with the polynomial $g(x)$ for which $f(x)=g(x)(x-\lambda_1)$. However, this is not true in general. 

\begin{exmpl}
Consider the real octonion algebra $\mathbb{O}$ with standard generators $i,j,\ell$, and the polynomial
$f(x)=i x^2+j x+\ell$.
It has two roots $\lambda_1=\frac{1}{2}(1+ij+i\ell+j\ell)$ and $\lambda_2=\frac{1}{2}(-1+ij-i\ell+j\ell)$.
Then $\lambda_1-\lambda_2=1+i\ell$, and $\gamma=(\lambda_1-\lambda_2) \lambda_2 (\lambda_1-\lambda_2)^{-1}=\frac{1}{2}(1+i\ell)\cdot \frac{1}{2}(-1+ij-i\ell+j\ell) \cdot (1-i \ell)=\frac{1}{2}(1+i\ell)(-1+j\ell)=\frac{1}{2}(-1-ij-i\ell+j\ell)$.
The polynomial $f(x)$ decomposes as $g(x)(x-\lambda_1)$ where $g(x)=ix+(j+i\lambda_1)=ix+\frac{1}{2}(i+j-\ell-(ij)\ell)$, whose only root is not $\gamma$.
On the other hand, $$h(x)=(\lambda_1-\lambda_2)^{-1} i x+((\lambda_1-\lambda_2)^{-1} i)\lambda_1+(\lambda_1-\lambda_2)^{-1}j$$
$$=\frac{1}{2}(i-\ell) x+\frac{1}{2}(-j-\ell)+\frac{1}{2}(j+(ij)\ell),$$
and then indeed $h(\gamma)=\frac{1}{2}(\ell-(ij)\ell)+\frac{1}{2}(-j-\ell)+\frac{1}{2}(j+(ij)\ell)=0$, and therefore $\gamma$ is a root of $h(x)$.
\end{exmpl}

So how do we keep factoring $g(x)$?
The answer is in the following section and involves the companion polynomial.

\section{The Companion Polynomial}

In \cite{Chapman:2020} the companion polynomial $C_f(x)$ was defined for any polynomial $f(x)=c_n x^n+\dots+c_1 x+c_0$ in the following way:
$$C_f(x)=b_{2n} x^{2n}+b_{2n-1}x^{2n-1}+\dots+b_1x+b_0$$
where for each $k \in \{0,\dots,2n\}$, if $k$ is odd then $b_k$ is equal to the sum of all the elements $\tr(\overline{c_i}c_j)$ with $0\leq i < j \leq n$ and $i+j=k$, and if $k=2m$ is even then $b_k$ is equal to the sum of all the elements $\tr(\overline{c_i}c_j)$ with $0\leq i < j \leq n$ and $i+j=k$ plus the element $\norm(c_m)$. 

\begin{rem}
Since $\tr(\overline{c_i}c_j)=\overline{c_i}c_j+\overline{c_j}c_i$ and $\norm(c_m)=\overline{c_m}c_m$, by a straight-forward computation (using the fact that $x$ is central) one sees that in fact 
$$C_f(x)=\overline{f(x)}f(x),$$
where $\overline{f(x)}=\overline{c_n} x^n+\dots+\overline{c_1}x+\overline{c_0}$.
\end{rem}

In \cite{Chapman:2020} it was proven that the roots of $C_f(x)$ coincide with the conjugacy classes of the roots of $f(x)$.
This leads us to the following fact:

\begin{thm}\label{factor}
A polynomial $f(x)$ of degree $n$ in $A[x]$ factors into linear factors
$$f(x)=((\dots(c(x-\lambda_n))\dots(x-\lambda_3))(x-\lambda_2))(x-\lambda_1)$$
 if and only if $C_f(x)$ decomposes as $C_f(x)=\norm(c)\cdot \prod_{k=1}^n (x^2-\tr(\gamma_k)x+\norm(\gamma_k))$ for some $\gamma_1,\dots,\gamma_n \in A$, in which case $\lambda_1,\dots,\lambda_n$ are inner conjugates of $\gamma_1,\dots,\gamma_n$, and represent the conjugacy classes of the roots of $f(x)$.
\end{thm}

\begin{proof}
Recall that $\lambda_1$ is a root of $f(x)$ if and only if $f(x)=g(x)(x-\lambda_1)$.
If that holds then
$C_f=\overline{f(x)}f(x)=(x-\overline{\lambda_1}) \overline{g(x)}g(x)(x-\lambda_1)=C_g(x)(x^2-\tr(\lambda_1)x+\norm(\lambda_1))$.
Since the roots of $C_g(x)$ are the conjugacy classes of the roots of $g(x)$,  we conclude that the conjugacy classes of the $f(x)$ are the union of the conjugacy classes of the roots of $g(x)$ and the conjugacy class of $\lambda_1$.

Thus, by induction, if $f(x)$ decomposes completely as the product of linear factors $f(x)=((\dots(c(x-\lambda_n))\dots(x-\lambda_3))(x-\lambda_2))(x-\lambda_1)$, then $C_f=\norm(c)\cdot \prod_{k=1}^n (x^2-\tr(\lambda_k)x+\norm(\lambda_k))$, the roots of $C_f$ are both the conjugacy classes of $\lambda_1,\dots,\lambda_n$, and the conjugacy classes of the roots of $f(x)$.
Consequently, the roots of $f(x)$ are the conjugacy classes of  $\lambda_1,\dots,\lambda_n$.

In the opposite direction, suppose $C_f(x)=\norm(c)\cdot \prod_{k=1}^n (x^2-\tr(\gamma_k)x+\norm(\gamma_k))$ for some $\gamma_1,\dots,\gamma_n \in A$.
Take $\lambda_1$ to be a root of $f(x)$ in the conjugacy class of $\gamma_1$.
Since for a root $\lambda_1$ of $f(x)$, $f(x)=g(x)(x-\lambda_1)$ and $C_f(x)=C_g(x)(x^2-\tr(\lambda_1)x+\norm(\lambda_1))$, we conclude $C_g(x)=\norm(c)\cdot \prod_{k=2}^n (x^2-\tr(\gamma_k)x+\norm(\gamma_k))$, and thus for each $k\in \{2,\dots,n\}$, there is a root of $g(x)$ in the conjugacy class of $\gamma_k$. Take in particular $\lambda_2$ to be the root of $g(x)$ in the conjugacy class of $\gamma_2$, and continue inductively.
\end{proof}

\begin{rem}
The analogous statement clearly holds true for quaternion algebras as well, i.e., if $Q$ is a quaternion algebra over $F$ and $f(x)$ is a polynomial in $Q[x]$, then $f(x)$ decomposes into linear factors as $f(x)=c(x-\lambda_n)\dots(x-\lambda_1)$ if and only if its companion polynomial $C_f(x)=\overline{f(x)}\cdot f(x)$ decomposes as $C_f(x)=\norm(c)\cdot \prod_{k=1}^n (x^2-\tr(\gamma_k)x+\norm(\gamma_k))$ for some $\gamma_1,\dots,\gamma_n \in Q$.
The argument is simpler here, because the algebra is associative: if $f(x)$ decomposes as $c(x-\lambda_n)\dots(x-\lambda_1)$, then $C_f(x)$ is by a straight-forward computation $\norm(c)\cdot \prod_{k=1}^n (x^2-\tr(\lambda_k)x+\norm(\lambda_k))$; and in the opposite direction, if $C_f(x)=\overline{f(x)}\cdot f(x)$ decomposes as $C_f(x)=\norm(c)\cdot \prod_{k=1}^n (x^2-\tr(\gamma_k)x+\norm(\gamma_k))$ for some $\gamma_1,\dots,\gamma_n \in Q$, then by \cite[Remark 3.5 and Theorem 3.6]{ChapmanMachen:2017}, $f(x)$ has roots $\lambda_1,\dots,\lambda_n$ where each $\lambda_k$ is in the same conjugacy class as $\gamma_k$.
\end{rem}

Note that in the decomposition $C_f(x)=\norm(c)\cdot \prod_{k=1}^n (x^2-\tr(\gamma_k)x+\norm(\gamma_k))$, each quadratic factor is central, and thus the product is well defined because it takes place in $F[x]$. The proof of Theorem \ref{factor} gives rise to an algorithm for factoring completely decomposable polynomials in $A[x]$:
\begin{algo}
Given a polynomial $f(x) \in A[x]$, in order to compute its factorization, 
\begin{enumerate}
\item Compute the companion polynomial $C_f(x)$.
\item Decompose $C_f(x)$ as $C_f=\norm(c)\cdot \prod_{k=1}^n (x^2-\tr(\gamma_k)x+\norm(\gamma_k))$ (on the way, make sure this decomposition exists, because it is necessary).
\item Take $\gamma_1$ and find a root $\lambda_1$ of the same conjugacy class for $f(x)$ by reducing the equation $f(\lambda_1)=0$ into a linear equation using the equality $\lambda_1^2-\tr(\gamma_1) \lambda_1+\norm(\gamma_1)\lambda_1=0$ (in a similar way to \cite[Theorem 3.4 and Algorithm 3.5]{Chapman:2020}).
\item Decompose $f(x)=f_2(x)(x-\lambda_1)$.
\item For $k\in \{2,\dots,n-1\}$, take $\gamma_k$ and find a root $\lambda_k$ of the same conjugacy class for $f_k(x)$ by reducing $f_k(x)=0$ into a linear equation, and decompose $f_k(x)=f_{k+1}(x)(x-\lambda_k)$.
\item We end up with a linear $f_n(x)$ that factors as $f_n(x)=c(x-\lambda_n)$. Then the obtained $\lambda_n,\dots,\lambda_1$ satisfy $f(x)=((\dots(c(x-\lambda_n))\dots(x-\lambda_3))(x-\lambda_2))(x-\lambda_1)$.
\end{enumerate}
\end{algo}

\begin{exmpl}
Consider the real octonion algebra $\mathbb{O}$ and the polynomial $f(x)=\ell x^3+i\ell x^2+\ell x+i\ell$.
Its companion polynomial is $C_f(x)=x^6+3x^4+3x^2+1=(x^2+1)^3$.
Therefore the roots of $f(x)$ are of trace $0$ and norm $1$.
They include $j$.
Take $\lambda_1=j$.
Then $f(x)=(\ell x^2+(\ell j+i\ell)x+(\ell j^2+i\ell j+\ell))(x-j)$, $g(x)=\ell x^2+(i-j)\ell x-(i j)\ell$.
Hence, $C_g(x)=(x^2+1)^2$, so its roots are still of trace $0$ and norm $1$.
Suppose $\lambda_2$ is a root of $g(x)$, then $\lambda_2^2=-1$, and so $0=g(\lambda_2)=-\ell+((i-j)\ell)\lambda_2-(ij)\ell$, which means $(1+ij)\ell=((i-j)\ell)\lambda_2$, and so $\lambda_2=\frac{1}{2}((1+ij)\ell)((i-j)\ell)=\frac{1}{2}((j-i)(1+ij))\ell^2=-j$.
Therefore, $g(x)=(\ell x+(i-j)\ell+\ell\cdot (-j))(x+j)=(\ell x+i\ell)(x+j)=(\ell(x+i))(x+j)$.
The obtained factorization of $f(x)$ is thus
$$f(x)=((\ell(x+i))(x+j))(x-j).$$
Note that in this very specific polynomial, we have $f(x)=\ell(x+i)(x^2-1)$, and thus there are plenty other possible ways of factoring it.
\end{exmpl}


\bibliographystyle{abbrv}

\end{document}